\documentclass[preprint,review,12pt]{elsarticle}
\usepackage{amsmath,epsfig,graphics}

\usepackage{amssymb}  
\newtheorem{theorem}{Theorem}[section]
\newtheorem{lemma}{Lemma}[section]
\newtheorem{remark}{Remark}[section]

\newtheorem{example}{Example}[section]

\newtheorem{problem}{Problem}[section]

\def\QED{~\rule[-1pt] {8pt}{8pt}\par\medskip}
\newenvironment{proof}{\noindent{\bf Proof: }}{\hspace*{\fill}\QED}

\def\QED{~\rule[-1pt]{5pt}{5pt}\par\medskip}

\def\ee{{\epsilon}}

\def\be{\begin{equation}}

\def\bi{\begin{itemize}}
\def\ee{\end{equation}}
\def\ei{\end{itemize}}

\def\z1{z^{-1}}
\def\la{\label}

\journal{ArXiv}

\begin{document}
\begin{frontmatter}
\title{On Convexification of Range Measurement Based Sensor and
Source Localization Problems}
\author{Bar\i\c{s} Fidan and Fatma Kiraz}

\address{School of Engineering, University of Waterloo, Waterloo, ON N2L 3G1\\ (\{fidan,fkiraz\}@uwaterloo.ca). Phone: +1 519 888-4567.}




\begin{abstract}
This manuscript is a preliminary pre-print version of a journal submission by the authors, revisiting the problem of range measurement based localization of a  signal source or a sensor. The major geometric difficulty of the problem comes from the non-convex structure of optimization tasks associated with range measurements, noting that the set of source locations corresponding to a certain distance measurement by a fixed point sensor is non-convex both in two and three dimensions. Differently from various recent approaches to this localization problem, all starting with a non-convex geometric minimization problem and attempting to devise methods to compensate the non-convexity effects, we suggest a geometric strategy to compose a convex minimization problem first, that is equivalent to the initial non-convex problem, at least in noise-free measurement cases. Once the convex equivalent problem is formed, a wide variety of convex minimization algorithms can be applied. The paper also suggests a gradient based localization algorithm utilizing the introduced convex cost function for localization. Furthermore, the effects of measurement noises are briefly discussed. The design, analysis, and discussions are supported by a set of numerical simulations.
\end{abstract}

\end{frontmatter}

\section{Introduction} \label{sec:Intro}
Over the last decade, there has been significant amount of studies on the problem of range or distance measurement based signal source/sensor localization \cite{TS05,Patwari05,hero,spawc,FDA08,MDD08,MFbook09,ZBbook12}. This problem is formulated in abstract terms in \cite{FDA08} as follows:

\begin{problem}\label{prob:main}
Given known 2 or 3- dimensional sensory station positions $x_1,\cdots,x_N$ ($N>2$ and $N>3$ in 2 and 3 dimensions respectively) and a signal source/target at unknown position $y^*$, estimate the value of $y^*$, from the measured  distances $d_i=\|y^*-x_i\|$.
\end{problem}
Problem \ref{prob:main} is defined in the form of a cooperative target/source localization task; nevertheless, it can be considered in the form of a sensor network node self-localization problem as well, where the $N$ stations represent $N$ anchors, and there is a $(N+1)$st sensor node at $y^*$ estimating its own position.

The major geometric difficulty of Problem \ref{prob:main} comes from the non-convex structure of optimization tasks associated with range measurements: The set of source locations corresponding to a certain distance measurement $d_i$ by a sensor located at point $x_i$ is non-convex both in two and three dimensions, in the form of a circle and a spherical shell, respectively. The generic attempt is then fusing all the distance measurements $d_1,\dots,d_N$ from the sensing points $x_1,\dots,x_N$, respectively, and finding the intersection of the non-convex source location sets $S(x_i,d_i)$ corresponding to the $(x_i,d_i)$ pairs. However, the non-convexity of these location sets limits the application of the algorithms devised based on intersection of the mentioned non-convex source location sets $S(x_i,d_i)$ and the corresponding non-convex cost functions.

This paper revisits Problem \ref{prob:main} following a different approach and suggests a geometric strategy to compose a convex geometric problem first, that is equivalent to the initially non-convex problem, at least in noise-free measurement cases. Once the convex equivalent problem is formed, a wide variety of convex minimization algorithms can be applied. The paper also suggests a gradient based localization algorithm based on the introduced convex cost function for localization. Furthermore, the effects of measurement noises are briefly discussed. The design, analysis, and discussions are supported by a set of numerical simulations.

The details of distance  measurement mechanisms used for the above problem is out of scope of this paper. Such details can be found, e.g., in \cite{MFbook09,ZBbook12}. Nevertheless, similar to \cite{FDA08}, for better visualization of the implementation of the localization task, we give here one mechanism example, received signal strength (RSS) approach: For a source
emitting a signal with source signal strength $A$ in a medium with power loss coefficient $\eta$, the RSS at a distance $d$ from the signal source is given by
\begin{equation} \la{eq:RSS}
s=A/d^\eta.
\end{equation}
Using (\ref{eq:RSS}), $d$ can be calculated given values of $A$, $s$, and $\eta$.

The rest of the paper is organized as follows: Section \ref{sec:Convex} introduces the proposed problem convexification strategy based on the notion of \emph{radical axis}. Section \ref{subsec:algorithm} proposes a gradient based localization algorithm minimizing the convex cost function introduced in Section \ref{sec:Convex}. Convergence analysis for the noise-free measurement cases is provided in Section \ref{subsec:analysis}. Simulation studies, including those testing the effects of measurement noises, are presented in Section \ref{sec:Simulations}. Closing remarks are given in Section \ref{sec:conc}.

\setcounter{equation}{0}
\section{Convexification of the Localization Problem} \label{sec:Convex}
\subsection{Non-Convex Cost Functions}
As stated in Section \ref{sec:Intro}, the approaches to Problem \ref{prob:main} in the literature start with a non-convex geometric minimization problem definition and attempt to devise methods to compensate the non-convexity effects. A typical natural selection of cost function to minimize \cite{FDA08} is
\begin{equation} \label{eq:Jnonconv}
J_1(y)=\frac{1}{2}\sum_{i=1}^N \lambda_i \left ( \|x_i-y\|^2-d_i^2\right )^2,
\end{equation}
where $\lambda_i$ ($i=1,\dots,N$) are positive weighting terms. A gradient localization algorithm based on minimization of the non-convex cost function (\ref{eq:Jnonconv}) has been proposed in \cite{FDA08}. Although this algorithm has proven stability and convergence properties, for these guaranteed properties to hold $y^*$ in Problem \ref{prob:main} is required to lie in a certain convex bounded region defined by the set $\{x_1,\dots,x_N\}$. Next we introduce a new cost function to overcome the aforementioned limitation.

\subsection{A Convex Cost Function Based on Radical Axes}
In two dimensions, if the distance measurements $d_i$ in Problem \ref{prob:main} are noise-free, the global minimizer of (\ref{eq:Jnonconv}) is located at $y^*$, where $J_1(y^*)=0$. Geometrically, $y^*$ is the intersection of the circles $C(x_i,d_i)$ with center $x_i$ and radius $d_i$. We re-formulate this later fact to form a convex cost function to replace the non-convex (\ref{eq:Jnonconv}), using the notion of \emph{radical axis}:

\begin{theorem} \label{thm:radax}\emph{(Fact 45 of \cite{Johnson07})} Given two non-concentric circles $C(c_1,r_1)$,$C(c_2,r_2)$, there is a unique line consisting of points $p$ holding equal powers with regard to these circles, i.e., satisfying
\[{\left\|p-c_1\right\|}^2-\ {r_1}^2={\left\|p-c_2\right\|}^2-\ {r_2}^2.\]
This line is perpendicular to the line connecting $c_1$ and $c_2$, and if the two circles intersect, passes through the intersection points.
\end{theorem}

The unique line mentioned in Theorem \ref{thm:radax} is called the \emph{radical axis} of $C(c_1,r_1)$ and $C(c_2,r_2)$ \cite{Johnson07}.

\begin{lemma}
\label{lem:ystar}
In 2 dimensions, if the distance measurements $d_i$ in Problem \ref{prob:main} are noise-free, the intersection set of the radical axes of any $N-1$ distinct circle pairs $C(x_i,d_i)$, $C(x_j,d_j)$ ($i\neq j$) is $\{y^*\}$.
\end{lemma}

\begin{proof}
The result straightforwardly follows from Problem \ref{prob:main} definition and the last statement of Theorem \ref{thm:radax}.
\end{proof}

In order to utilize Lemma \ref{lem:ystar}, we first derive the mathematical representation of the radical axis $l_{ij}$ of a circle pair $C(x_i,d_i)$, $C(x_j,d_j)$ ($i\neq j$) given the values of $x_i,x_j,d_i,d_j$. Such a radical axis line is illustrated in Figure \ref{fig:radaxline}.
\begin{figure}[htb]
\begin{center}
\includegraphics[scale=.5]{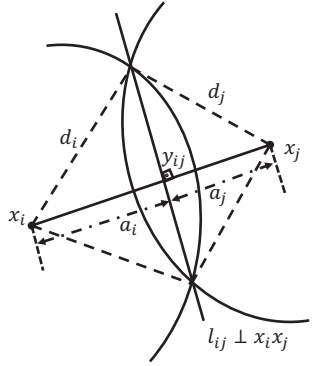}
\caption{Radical axis of a circle pair $C(x_i,d_i)$, $C(x_j,d_j)$.}
\label{fig:radaxline}
\end{center}
\end{figure}
$l_{ij}$ perpendicularly intersects $x_ix_j$ at $y_{ij}$. Hence any point $y$ on it satisfies
\begin{equation} \label{eq:lij}
{(y-y_{ij})}^Te_{ij}=0,
\end{equation}
where
\[e_{ij}=x_{j}-x_i.\]
It can be observed from Figure \ref{fig:radaxline} that
\begin{equation} \label{eq:yij}
y_{ij}=x_i+a_i\frac{e_{ij}}{\|e_{ij}\|},
\end{equation}
as well as ${d_i}^2-{a_i}^2={d_j}^2-{a_j}^2={d_j}^2-(\|e_{ij}\|-a_i)^2$,
from which $a_i$ can be calculated as
\begin{equation} \label{eq:ai0}
a_i=\frac{{\|e_{ij}\|^2+{d_i}^2-d^2_{j}}}{2\|e_{ij}\|}.
\end{equation}
The equations (\ref{eq:lij})--(\ref{eq:ai0}) form the explicit mathematical representation we were looking for.

Next, we focus on utilization of Lemma \ref{lem:ystar} to compose a convex alternative for (\ref{eq:Jnonconv}). Leaving the optimal selection of the $N-1$ distinct circle (or corresponding sensor node) pairs to a future study, we consider a sequential pair selection for the rest of this paper: For each $i\in\{1,...,N-1\}$, let pair $i$  denote the circle pair $C(x_i,d_i)$, $C(x_{i+1},d_{i+1})$, $l_i$ denote the corresponding radical axis, $y_i$ denote the intersection of $l_i$ and $x_ix_{i+1}$; and according ly let us use the following special case of \eqref{eq:yij},\eqref{eq:ai0}:
\begin{eqnarray}
y_{i}&=&x_i+a_i\frac{e_{i}}{\|e_{i}\|},~~e_{i}~=~x_{i+1}-x_i, \label{eq:yi}\\
a_i&=&\frac{{\|e_{i}\|^2+{d_i}^2-d^2_{i+1}}}{2\|e_{i}\|}. \label{eq:ai}
\end{eqnarray}
Lemma \ref{lem:ystar} implies that the intersection set of $l_1,\dots,l_{N-1}$ is $\{y^*\}$, i.e., $y^*$ is the unique point satisfying
\begin{equation} \label{eq:li}
{(y-y_{i})}^Te_{i}=0
\end{equation} for all $i\in\{1,...,N-1\}$. Hence $y^*$ is the unique solution of the equation $J(y)=0$, where
\begin{equation} \label{eq:J}
J\left(y\right)=\frac{1}{2}\sum^{N-1}_{i=1}{({{\left(y-y_i\right)}^Te_i)}^2}.
\end{equation}
Observing that (\ref{eq:J}) is a convex cost function, we reach the following result:

\begin{lemma}
\label{lem:J}
In 2 dimensions, if the distance measurements $d_i$ in Problem \ref{prob:main} are noise-free, then $y^*$ satisfies the following:
\begin{enumerate}
\item  The radical axis lines $l_1,\dots\l_{N-1}$ defined by (\ref{eq:li}) for $i=1,\dots,N-1$, respectively, intersect at $y^*$.
\item  $y^*$ is the unique local minimizer and, hence, the global minimizer of (\ref{eq:J}).
\end{enumerate}
\end{lemma}

\setcounter{equation}{0}
\section{Localization Algorithm} \label{sec:Gradient}
\subsection{The Algorithm} \label{subsec:algorithm}
As indicated above, a variety of adaptive algorithms to reach $y^*$ in Problem \ref{prob:main} can be devised based on the convex cost function (\ref{eq:J}). Here, to accommodate easy analysis and comparison with \cite{FDA08}, we propose a gradient based adaptive algorithm, with the standard iterations of estimate updates in the negative gradient direction
\begin{equation}\label{eq:gradJ}
-\nabla J(y)=-\left (\frac{\partial J(y)}{\partial y}\right )=-\sum^{N-1}_{i=1}{\left[{\left(y-y_i\right)}^Te_i\right]e_i}.
\end{equation}
The corresponding gradient based adaptive localization algorithm is given by
\begin{equation} \label{eq:algorithm}
y\left[k+1\right]=y\left[k\right]-\mu \nabla J(y[k]),
\end{equation}
where $\mu$ is a small positive design coefficient.
\subsection{Convergence Anaylsis} \label{subsec:analysis}
Because of Lemma \ref{lem:J}, $\nabla J(y)=0$ if and only if $y=y^*$, and hence the algorithm (\ref{eq:algorithm}) settles at $y[k]=y^*$ once it reaches that point. Further discussions on selection of the gradient gain $\mu$ and its effects on the convergence of the gradient descent algorithm \eqref{eq:algorithm} can be found in many convex optimization books, see, e.g., \cite{Bertsekas,Polyak,Nesterov}. Here, we summarize the main convergence result for \eqref{eq:algorithm} with constant gain $\mu>0$:
\begin{theorem} \label{thm:conv}
Consider Problem \ref{prob:main} in 2 dimensions, with noise-free distance measurements $d_i$. Assume that the $x_i$, $i\in\{1,\dots,N\}$ are non-collinear. Then for every $M>0$, there exists a $ \mu^*(M)$ such that
\begin{equation} \label{eq:lim_y}
\lim_{k\rightarrow\infty}y[k]=y^*,
\end{equation}
whenever
\be \la{ybd} J(y[0]) \leq M \ee and \be \la{mbd} 0<\mu \leq
\mu^*(M), \ee
i.e., for any selection of $y[0]$, there exists a sufficiently small $\mu$ for the algorithm \eqref{eq:algorithm} such that \eqref{eq:lim_y} is satisfied.
\end{theorem}

\begin{proof}
Consider an arbitrary scalar $M>0$ and the corresponding convex st $S_J(M)=\{y:~J(y)\leq M\}$. Since \eqref{eq:J} is a quadratic and hence differentiable function of $y$, it is Lipschitz within $S_J(M)$, i.e., there exist $L(M)$ such that for any $\bar{y}_1, \bar{y}_2\in S_J(M)$,
$\|\nabla J(\bar{y}_1)-\nabla J(\bar{y}_2)\|\leq L(M)$. Hence, Proposition 1.2.3 of \cite{Bertsekas} (or Theorem 1 on pp.~21 of \cite{Polyak}) implies that, for $y[0]\in S_J(M)$ and
\be
0<\mu<\mu^*(M)=2/L(M),
\ee
we have
(i) $J(y[k+1])\leq J(y[k])$, and hence $y[k]\in S_J(M)$ for $k=0,1,\dots$; and (ii)
$\lim_{k\rightarrow\infty}\nabla J(y[k])=0$, which, by Lemma \ref{lem:ystar} is equivalent to \eqref{eq:lim_y}.
\end{proof}

\begin{remark}
Theorem \ref{thm:conv} implies that, in 2-D, the proposed algorithm \eqref{eq:algorithm} has global convergence property for noise-free measurement cases. Hence, using \eqref{eq:algorithm} in place of \eqref{eq:algorithm_old} of \cite{FDA08}, the limitations imposed by Theorems 3.1, 3.2, 4.3 of \cite{FDA08} on location of $y^*$ will be circumvented.
\end{remark}

\setcounter{equation}{0}
\section{Simulations} \label{sec:Simulations}
In this section, we present a summary of our simulations studies testing the performance of the localization algorithm proposed in the previous section. To accommodate easy and fair comparisons with the algorithm proposed in \cite{FDA08} and other algorithms in the literature considered in \cite{FDA08}, we present the results for the same sample settings in that paper. For all the simulations below, the step size for algorithms \eqref{eq:algorithm} and \eqref{eq:algorithm_old} is selected as $\mu=0.001$, similarly to \cite{FDA08}.

Examples 2.1 and 2.2 of \cite{FDA08} provide two noise free cases having false stationary points for the cost function \eqref{eq:Jnonconv}. Example 2.1 of \cite{FDA08} has a single false stationary point, which is unstable. Hence, for this example, the gradient adaptive law
\begin{equation} \label{eq:algorithm_old}
y\left[k+1\right]=y\left[k\right]-\mu \nabla J_1(y[k]),
\end{equation}
using the non-convex cost function \eqref{eq:Jnonconv} is guaranteed to converge to the actual target position $y^*$ for sufficiently small $\mu$. Convergence is not guaranteed for Example 2.2 of \cite{FDA08}, on the other hand, since there is a stable false stationary point at $[3,3]^T$ in this case, i.e., for a certain set of initial estimates $y[0]$, the algorithm \eqref{eq:algorithm_old} will converge to $[3,3]^T$ instead of $y^*=0$. It can be easily seen, based on Theorem \ref{thm:conv} that this is not the case for \eqref{eq:algorithm}, i.e., \eqref{eq:algorithm} converges to $y^*$ for both Examples 2.1 and 2.2 of \cite{FDA08}. The example below visualizes the second case.

\begin{example} \label{ex:FalseStationary} Consider Example 2.2 of \cite{FDA08}, i.e. Problem \ref{prob:main} with
$x_1=[1,1]^T$, $x_2=[1,3]^T$, $x_3=[3,1]^T$, and $y^*=0$. In this case $d_3^2=d_2^2=10$
and $d_1^2=2$. It was noted in \cite{FDA08} that $y^*_s=[3,3]^T$ is a locally stable spurious stationary
point in this case for the cost function \eqref{eq:Jnonconv} with $\lambda_i=1$. A case where the estimation algorithm \eqref{eq:algorithm_old} produces estimates converging to $y^*_s$ instead of $y^*$, even with noiseless distance measurements is observed for initial estimate $y[0]=[3,2]^2$, as shown by the blue dashed curve in Figure \ref{fig_Ex2dim_nonconv}. The black solid curve in the same figure shows the result using the proposed algorithm \eqref{eq:algorithm} for the same case and the same initial estimate,w hcih converges to the actual source position $y^*$.
\begin{figure}[htb]
\begin{center}
\includegraphics[scale=.45]{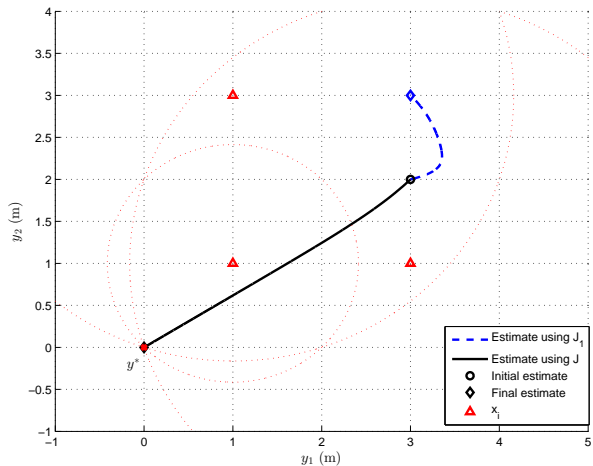}
\caption{Example \ref{ex:FalseStationary} with algorithm \eqref{eq:algorithm_old} of \cite{FDA08} (blue dashed) and the prosed algorithm \eqref{eq:algorithm} (black solid).}
\label{fig_Ex2dim_nonconv}
\end{center}
\end{figure}
\end{example}

The next example has the same settings as the one studied in Section V of \cite{FDA08} with results summarized Figure 10 of that paper, with noisy RSS based measurement settings, considering continuous source emission. These settings were used in \cite{FDA08} to compare the performance of the gradient algorithm \eqref{eq:algorithm_old} with maximum likelihood estimation and Projection on Convex Sets (POCS) algorithms \cite{hero,spawc}. Due to space limitations and since a comparison of \eqref{eq:algorithm_old} with the other aforementioned algorithms was already provided in \cite{FDA08}, here we present only comparison of \eqref{eq:algorithm} and \eqref{eq:algorithm_old}.

\begin{example} \label{ex:TrSP_SecVB}
Consider a network of four RSS sensors located at $[-2,-1]^T$, $[-1,-3]^T$, $[-1,1]^T$, $[1,0]^T$, subject to the signal model
\begin{equation}\label{eq:RSS_shadow}
    s=\omega_sA/d^\eta,
\end{equation}
where $\omega_s[{\rm dB}]\triangleq 10 \log \omega_s$ is a
zero-mean Gaussian noise, i.e. $\omega_s[{\rm
dB}]\sim\mathcal{N}(0,\sigma_s^2)$. \eqref{eq:RSS_shadow} includes a log-normal shadowing term $\omega_s$, as the dominant measurement noise source, in the RSS model \eqref{eq:RSS}. Let $\eta=3$. For this setting, the algorithms \eqref{eq:algorithm} and \eqref{eq:algorithm_old} have been comparatively tested for the shadowing noise standard deviation $\sigma_s$ varying from 0 to 5~dB. The mean square estimation errors for the 1000 randomly selected initial estimates and target positions (shown in Figure \ref{fig:RandSettings}) are shown in Figures \ref{fig:ein} and \ref{fig:eout}. Figure \ref{fig:ein} covers the cases where the target selections are inside the ellipse, and hence convergence of \eqref{eq:algorithm_old} is guaranteed for $\sigma_s=0$, while Figure \ref{fig:eout} covering the cases with target selections outside this ellipse. As can be seen in these figures, in addition to particular ill-conditioned cases such as the one described in Example \ref{ex:FalseStationary}, the proposed algorithm is effective in generic settings, especially when the shadowing noise is small.

\begin{figure}[htb]
\begin{center}
\includegraphics[scale=.45]{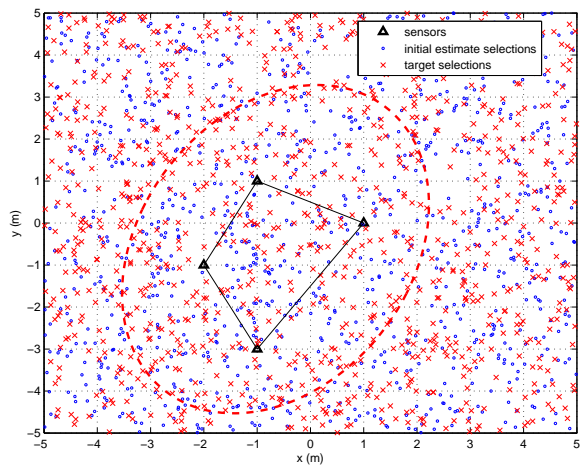}
\caption{Random initial condition and target selections for Example \ref{ex:TrSP_SecVB}. The ellipse is the one defined by Theorem 3.3. of \cite{FDA08}. For targets inside this ellipse, the algorithm \eqref{eq:algorithm_old} is guaranteed to converge with noise-free measurements.}
\label{fig:RandSettings}
\end{center}
\end{figure}

\begin{figure}[htb]
\begin{center}
\includegraphics[scale=.45]{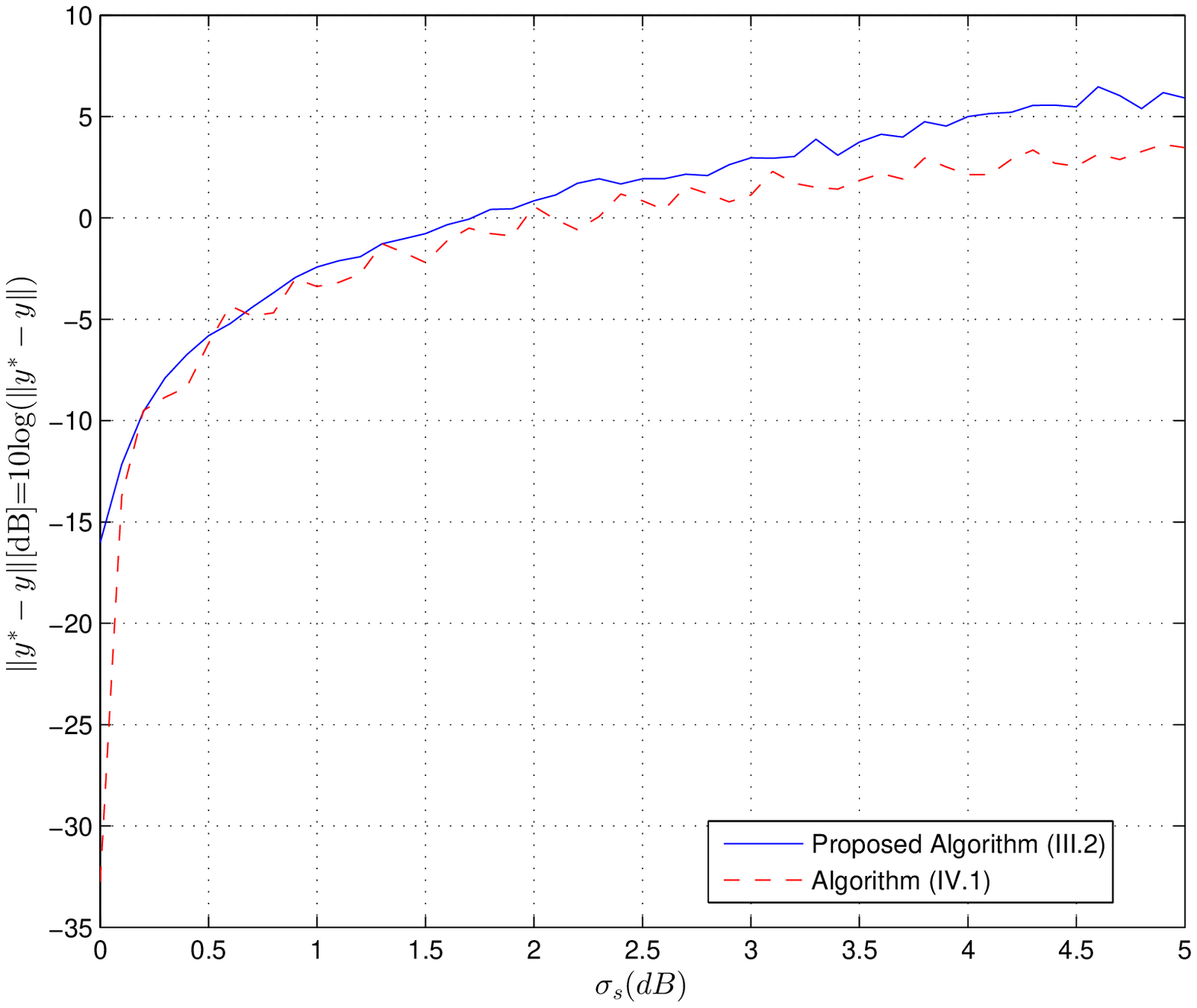}
\caption{Average estimation error versus shadowing noise standard deviation, with targets inside the ellipsoid region guaranteing convergence of \eqref{eq:algorithm_old}.}
\label{fig:ein}
\end{center}
\end{figure}

\begin{figure}[htb]
\begin{center}
\includegraphics[scale=.45]{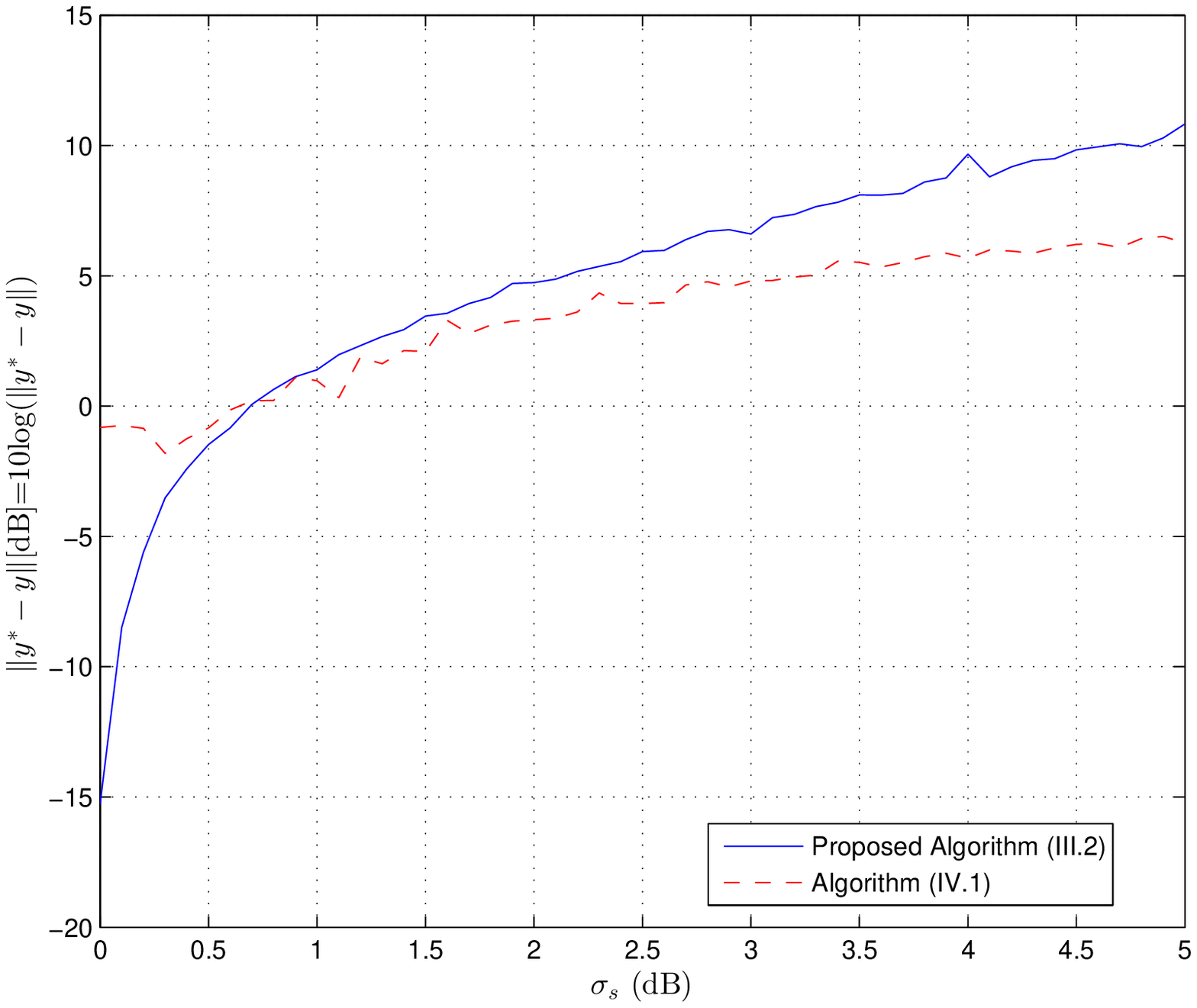}
\caption{Average estimation error versus shadowing noise standard deviation, with targets outside the ellipsoid region guaranteing convergence of \eqref{eq:algorithm_old}.}
\label{fig:eout}
\end{center}
\end{figure}

\end{example}

\section{Conclusion} \label{sec:conc}
We have proposed a geometric strategy to define range measurement based sensor network localization of a signal source/target as a convex optimization problem. Based on this strategy we have developed a gradient based localization algorithm, which is globally convergent for noise-free measurements. In addition to formal convergence analysis, the use of the proposed strategy and algorithm is demonstrated via a set of numerical simulations. The convexity and global convergence aspects of the proposed methodology make it a good candidate for implementation, especially in precise sensor network localization tasks involving multiple targets in a wide region of interest and high precision range sensors.


\end{document}